\documentclass[12pt, a4paper]{article}

\usepackage{amsmath}                                                            
\usepackage{amssymb}                                                            
\usepackage{amsfonts}                                                           
\usepackage{ifthen}                                                             
\usepackage{type1cm}                                                            
\usepackage{tikz}               % drawings

\newtheorem{theorem}{Theorem}
\newtheorem{lemma}{Lemma}
\newtheorem{proposition}{Proposition}
\newtheorem{corollary}{Corollary}
\newtheorem{definition}{Definition}
\newtheorem{observation}{Observation}

\newenvironment{proof}{\noindent{\bf Proof.}}{$\hfill \Box$ \vspace{10pt}}

\newcommand{\RR}{\mathbb{R}}

\DeclareMathAlphabet{\masc}{U}{eus}{m}{n}                                       
\DeclareMathAlphabet{\mafr}{U}{euf}{m}{n}                                       
\DeclareMathAlphabet{\malf}{OT1}{cmtt}{m}{it}

\newcommand{\Span}{\operatorname*{span}}
\newcommand{\Ker}{\operatorname*{ker}}
\newcommand{\Eig}{\operatorname*{Eig}}

\newcommand{\TM}{{\mafr T}}

\newcommand{\cvec}[1]{{\left(\begin{array}{c}#1\end{array}\right)}}
\newcommand{\textall}{\ \text{for all}\ }

\DeclareMathOperator{\FSud}{FSud}
\newcommand{\FSudUp}[1]{\FSud^{\uparrow #1}}
\newcommand{\Up}[1]{{#1}^{\uparrow}}
\newcommand{\jv}{{\malf 1}}
\newcommand{\Bb}{\mathcal{B}}
\newcommand{\Xb}{\mathcal{X}}
\newcommand{\Eb}{\mathcal{E}}
\newcommand{\Kb}{\mathcal{K}}
                           
\allowdisplaybreaks

\newcommand{\ab}{\allowbreak}

\begin{document}

\title{Spectrum of free-form Sudoku graphs \\ \large (Extended preprint)}

\author{  Mohammad Abudayah \and Omar Alomari  \\
					\emph{mohammad.abudayah@gju.edu.jo} \\
					\emph{omar.alomari@gju.edu.jo}\\
					\emph{German Jordanian University} \\
					\emph{Amman, Jordan} \\\\
					Torsten Sander\footnote{Corresponding author} \\ \emph{t.sander@ostfalia.de​}\\
					\emph{Fakult\"at f\"ur Informatik,} \\
					\emph{Ostfalia Hochschule f\"ur angewandte Wissenschaften} \\
					\emph{Wolfenb\"uttel, Germany}\\
   }

\maketitle

\thispagestyle{empty}

\centerline{\large \bf Abstract}

A free-form Sudoku puzzle is a square arrangement of $m\times m$ cells
such that the cells are partitioned into $m$ subsets (called blocks) of equal cardinality.
The goal of the puzzle is to place integers $1,\ldots,m$ in the cells such that 
the numbers in every row, column and block are distinct. %Typically, some cells are pre-filled as `clues'.
Represent each cell by a vertex and add edges between two vertices exactly when the corresponding
cells, according to the rules, must contain different numbers. This yields the associated free-form Sudoku graph.
This article studies the eigenvalues of free-form Sudoku graphs, most notably integrality.
Further, we analyze the evolution of eigenvalues and eigenspaces of such graphs
when the associated puzzle is subjected to a `blow up' operation, scaling the cell grid including its block partition.


%05C50 Graphs and matrices
%15A18 Eigenvalues, singular values, and eigenvectors 

%-------------------------------------------------------------------------------------------------------
\section{Introduction}
%-------------------------------------------------------------------------------------------------------

The recreational game of Sudoku has been popular for several years now. Its classic variant is played
on a board with $9\times 9$ cells, subdivided into a $3\times 3$ grid of square blocks containing $3\times 3$ cells each.
Each cell may be empty or contain one of the numbers $1,\ldots,9$. 
A number of cells of each puzzle have been pre-filled by the puzzle creator. The
goal of the puzzle solver is to fill the remaining cells with the numbers $1,\ldots,9$ 
such that in the completed puzzle the number of each cell occurs only once per row, column and block.
An example is shown in Figure \ref{figsudoku}.

Let us call the classical variant the $3$-Sudoku. Its $3^2\times 3^2$ board contains $3^4$ cells.
One can readily generalize the game to $n$-Sudokus by playing on $n^2\times n^2$ board with $n^4$ cells,
subdivided into $n^2$ square blocks with $n^2$ cells each. 
The permitted numbers in the cells now range from $1,\ldots,n^2$, but
the remaining restrictions for a valid solution remain the same. The $2$-Sudoku is also known as a Shidoku.

\begin{figure}%[ht]
\begin{center}
\resizebox{0.5\hsize}{!}{%
\begin{tikzpicture}
\draw[step=1cm, gray] (0,0) grid (9,9);
\draw[step=3cm, very thick] (0,0) grid (9,9);

%8,7,6
\node at (0.5,8.5) {\bf 8};\node at (1.5,8.5) {4}			;\node at (2.5,8.5) {5};
\node at (3.5,8.5) {6}		;\node at (4.5,8.5) {\bf 1}	;\node at (5.5,8.5) {\bf 9};
\node at (6.5,8.5) {\bf 2};\node at (7.5,8.5) {7}			;\node at (8.5,8.5) {3};

\node at (0.5,7.5) {\bf 1};\node at (1.5,7.5) {6}			;\node at (2.5,7.5) {\bf 7};
\node at (3.5,7.5) {\bf 3};\node at (4.5,7.5) {4}			;\node at (5.5,7.5) {2};
\node at (6.5,7.5) {9}		;\node at (7.5,7.5) {8}			;\node at (8.5,7.5) {5};

\node at (0.5,6.5) {\bf 2};\node at (1.5,6.5) {\bf 3}	;\node at (2.5,6.5) {\bf 9};
\node at (3.5,6.5) {\bf 8};\node at (4.5,6.5) {\bf 7}	;\node at (5.5,6.5) {5};
\node at (6.5,6.5) {\bf 4};\node at (7.5,6.5) {\bf 1}	;\node at (8.5,6.5) {\bf 6};

% 5,4,3
\node at (0.5,5.5) {\bf 9};\node at (1.5,5.5) {\bf 2}	;\node at (2.5,5.5) {4};
\node at (3.5,5.5) {\bf 1};\node at (4.5,5.5) {3}			;\node at (5.5,5.5) {\bf 8};
\node at (6.5,5.5) {\bf 5};\node at (7.5,5.5) {6}			;\node at (8.5,5.5) {\bf 7};

\node at (0.5,4.5) {\bf 3};\node at (1.5,4.5) {7}			;\node at (2.5,4.5) {\bf 1};
\node at (3.5,4.5) {2}		;\node at (4.5,4.5) {5}			;\node at (5.5,4.5) {\bf 6};
\node at (6.5,4.5) {8}		;\node at (7.5,4.5) {\bf 4}	;\node at (8.5,4.5) {9};

\node at (0.5,3.5) {\bf 6};\node at (1.5,3.5) {\bf 5}	;\node at (2.5,3.5) {8};
\node at (3.5,3.5) {\bf 4};\node at (4.5,3.5) {9}			;\node at (5.5,3.5) {\bf 7};
\node at (6.5,3.5) {1}		;\node at (7.5,3.5) {\bf 3}	;\node at (8.5,3.5) {\bf 2};

% 2,1,0
\node at (0.5,2.5) {\bf 5};\node at (1.5,2.5) {\bf 1}	;\node at (2.5,2.5) {\bf 6};
\node at (3.5,2.5) {\bf 7};\node at (4.5,2.5) {2}			;\node at (5.5,2.5) {\bf 4};
\node at (6.5,2.5) {3}		;\node at (7.5,2.5) {9}			;\node at (8.5,2.5) {\bf 8};

\node at (0.5,1.5) {4}		;\node at (1.5,1.5) {8}			;\node at (2.5,1.5) {2};
\node at (3.5,1.5) {9}		;\node at (4.5,1.5) {\bf 6}	;\node at (5.5,1.5) {\bf 3};
\node at (6.5,1.5) {7}		;\node at (7.5,1.5) {\bf 5}	;\node at (8.5,1.5) {\bf 1};

\node at (0.5,0.5) {7}		;\node at (1.5,0.5) {\bf 9}	;\node at (2.5,0.5) {\bf 3};
\node at (3.5,0.5) {\bf 5};\node at (4.5,0.5) {\bf 8}	;\node at (5.5,0.5) {1};
\node at (6.5,0.5) {\bf 6};\node at (7.5,0.5) {\bf 2}	;\node at (8.5,0.5) {4};

\end{tikzpicture}
}
\end{center}
\caption{Example Sudoku puzzle}\label{figsudoku}
\end{figure}

Despite the seemingly recreational character of the game, it offers a surprising number of mathematical facets.
This makes Sudoku an interesting topic for mathematics lessons \cite{EM}, \cite{zbMATH05932384}.
Among the topics touched by Sudoku are problem solving, latin squares, counting, exhausting symmetry and
coloring problems on graphs (see, for example, the introductory book \cite{takingsud}).
Due to this, more and more interesting results about Sudoku have been published in the recent past.
For instance, different approaches for solvers have been
presented in \cite{zbMATH06759817}, \cite{zbMATH06052119}, \cite{zbMATH06142305}.
The combinatorial properties of completed Sudoku squares as a family of Latin squares, in particular the search for
orthogonal pairs have been considered in \cite{zbMATH06751283}, \cite{zbMATH06619799}, \cite{zbMATH06322940}, \cite{zbMATH06092025}.
Other researchers focus on algebraic aspects of Sudoku, especially groups and rings associated with Sudoku
(cf.\ \cite{zbMATH06475256}, \cite{zbMATH06288901}, \cite{zbMATH06172103}).

One of the most intuitive links between Sudoku and mathematics is that the processes of
solving a Sudoku can be interpreted as completing a given partial vertex coloring of
a certain graph (each preassigned color corresponds to a prefilled number in a cell, a so-called clue). 
To this end, we represent each cell of the given Sudoku square by a single
vertex. Two vertices are adjacent if and only if the associated cells must not contain the
same number (according to the rules of the game). 
Figure \ref{figshidokuconstr} (a) depicts the
adjacencies in the neighborhood of an exemplary vertex of the Shidoku graph.
Using this representation, the valid solutions of a given
$n$-Sudoku can be characterized as proper optimal vertex colorings (using $n^2$ colors each)
of the associated Sudoku graph.

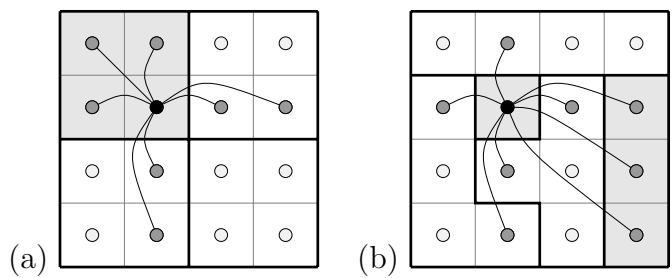
\begin{figure}%[ht]
\begin{center}(a)
\resizebox{0.25\hsize}{!}{%
\begin{tikzpicture}
\draw[fill=gray!20]    (0,2) -- ++(2,0) -- ++(0,2) -- ++(-2,0); % shading first
\draw[step=1cm, gray] (0,0) grid (4,4);
\draw[step=2cm, very thick] (0,0) grid (4,4);
\draw (1.5,2.5) .. controls (1.25,3.0) .. (1.5,3.5); 	% neighbor above
\draw (1.5,2.5) .. controls (2,2.75) .. (2.5,2.5); 		% neighbor to the right
\draw (1.5,2.5) .. controls (2.2,3.0) .. (3.5,2.5);		% neighbor to the far right
\draw (1.5,2.5) .. controls (1.0,2.75) .. (0.5,2.5);	% neighbor to the left
\draw (1.5,2.5) .. controls (1.25,2.0) .. (1.5,1.5);	% neighbor below
\draw (1.5,2.5) .. controls (1.0,1.8) .. (1.5,0.5);		% neighbor far below
\draw (1.5,2.5) -- (0.5,3.5);													% neighbor left above
% the vertices in general
\foreach \y in {0,...,3} \foreach \x in {0,...,3} \draw[fill=gray!10] (\x.5,\y.5) circle (0.1cm);
% considered vertex
\draw[fill] (1.5,2.5) circle (0.1cm);
% its neighbors
\draw[fill=gray!80] (1.5,3.5) circle (0.1cm); 
\draw[fill=gray!80] (2.5,2.5) circle (0.1cm); 
\draw[fill=gray!80] (3.5,2.5) circle (0.1cm); 
\draw[fill=gray!80] (0.5,2.5) circle (0.1cm); 
\draw[fill=gray!80] (1.5,1.5) circle (0.1cm); 
\draw[fill=gray!80] (1.5,0.5) circle (0.1cm); 
\draw[fill=gray!80] (0.5,3.5) circle (0.1cm); 
\end{tikzpicture}}\hspace*{0.5cm}(b)
\resizebox{0.25\hsize}{!}{%
\begin{tikzpicture}
\draw[fill=gray!20]    (1,2) -- ++(1,0) -- ++(0,1) -- ++(-1,0); % shading first
\draw[fill=gray!20]    (3,0) -- ++(1,0) -- ++(0,3) -- ++(-1,0); % shading first
\draw[step=1cm, gray] (0,0) grid (4,4);
\draw[step=4cm, very thick] (0,0) grid (4,4);
\draw[very thick] (2,0) -- (2,1) -- (1,1) -- (1,3) -- (0,3);
\draw[very thick] (1,3) -- (4,3);
\draw[very thick] (3,0) -- (3,3);
\draw[very thick] (1,2) -- (2,2) -- (2,3);
\draw (1.5,2.5) .. controls (1.25,3.0) .. (1.5,3.5); 	% neighbor above
\draw (1.5,2.5) .. controls (2,2.75) .. (2.5,2.5); 		% neighbor to the right
\draw (1.5,2.5) .. controls (2.2,3.0) .. (3.5,2.5);		% neighbor to the far right
\draw (1.5,2.5) .. controls (1.0,2.75) .. (0.5,2.5);	% neighbor to the left
\draw (1.5,2.5) .. controls (1.25,2.0) .. (1.5,1.5);	% neighbor below
\draw (1.5,2.5) .. controls (1.0,1.8) .. (1.5,0.5);		% neighbor far below
\draw (1.5,2.5) .. controls (1.75,1.85) .. (3.5,0.5);	% neighbor in lower right corner
\draw (1.5,2.5) .. controls (2.0,2.5) .. (3.5,1.5);		% neighbor above the previous one
% the vertices in general
\foreach \y in {0,...,3} \foreach \x in {0,...,3} \draw[fill=gray!10] (\x.5,\y.5) circle (0.1cm);
% considered vertex
\draw[fill] (1.5,2.5) circle (0.1cm);
% its neighbors
\draw[fill=gray!80] (1.5,0.5) circle (0.1cm); 
\draw[fill=gray!80] (1.5,1.5) circle (0.1cm); 
\draw[fill=gray!80] (1.5,3.5) circle (0.1cm); 
\draw[fill=gray!80] (0.5,2.5) circle (0.1cm); 
\draw[fill=gray!80] (2.5,2.5) circle (0.1cm); 
\draw[fill=gray!80] (3.5,2.5) circle (0.1cm); 
\draw[fill=gray!80] (3.5,1.5) circle (0.1cm); 
\draw[fill=gray!80] (3.5,0.5) circle (0.1cm); 
\end{tikzpicture}}
\end{center}
\caption{Deriving the graph of a classic and a free-form Shidoku puzzle}\label{figshidokuconstr}
\end{figure}

It is possible to investigate many questions about Sudoku with
the help of graph theory. For example, there has been some research interest in
counting the number of possible completions of a partially colored Sudoku graph. 
A related question is how
to (minimally) partially color the graph such that there is only one way to complete the
coloring to a valid Sudoku solution \cite{zbMATH06246071}. It has even been shown
that $16$ precolored vertices do not  suffice in order to guarantee a unique completion,
however by means of computational brute force \cite{zbMATH06328244}. Not only the number
of clues is a topic of interest but also where to place them \cite{zbMATH05823702}.
Other notions of coloring have been applied to Sudoku graphs as well, for example list coloring \cite{zbMATH05994814}.

Sudoku graphs have also been studied from the perspective of algebraic graph theory.
It has been shown that every Sudoku graph is integral, i.e.\ all eigenvalues of its adjacency
matrix are integers. This has been shown algebraically by means of group characters \cite{semester}.
But it also follows from the fact that Sudoku graphs are actually NEPS of complete graphs, 
hence they belong to the class of gcd-graphs, a subclass of the integral Cayley graphs over abelian groups \cite{sudejc}.
Here, NEPS is the common short form of the non-complete extended $p$-sum, a generalized graph product
that includes many known products \cite{zbMATH03908461}.

There exist many generalizations and variants of the classic Sudoku, the $n$-Sudokus being the most
common one. Besides changing the size of a Sudoku one can also try to vary every other aspect of 
a Sudoku, e.g.\ change the rules, introduce additional rules or change the shape of the blocks.
The book \cite{takingsud} presents many such variants.
Changing the shape of the blocks leads to the notion of a {\em free-form Sudoku}. 
Given a square arrangement of $m\times m$ cells (where $m$ need not be a square number any more), 
we permit the blocks to be an arbitrary partition $T$ of the cells into $m$ subsets of equal cardinality.
Note that the blocks are not required to be contiguous arrangements of cells.
The Sudoku rules remain unchanged, in the sense that we need to fill in the numbers $1, \ldots, m$ in
such that each row, column and block contains $m$ distinct numbers. 
We shall denote the associated graph by $\FSud(m, T)$.
Figure \ref{figshidokuconstr} (b) illustrates the
adjacencies in the neighborhood of an exemplary vertex for the graph of a free-form Shidoku.
The cells of the (non-contiguous) block containing the considered vertex have been shaded.

It seems that free-form Sudokus have not been studied in the  literature so far although they are included in
many Sudoku puzzle books that offer challenging variants.  Computer experiments readily indicate
that integrality of the eigenvalues critically depends on the degree of symmetry exhibited by the chosen
cell partition. However, free-form Sudoku graphs are well-structured enough to exactly predict 
the changes of their spectra and eigenspaces when they are transformed in certain ways.
%The purpose of this article is to study how the eigenvalues and eigenspaces of free-form Sudoku graphs evolve
%when the associated puzzle is subjected to an operation that scales the cell grid including its arrangement 
%(we call this a {\em blow up}). 

%-------------------------------------------------------------------------------------------------------
\section{Integrality}
%-------------------------------------------------------------------------------------------------------

In the following, consider a given free-form Sudoku graph $\FSud(m, T)$ and let $A$ be its adjacency matrix
(with respect to some arbitrary but fixed vertex order).
The structure of the adjacency matrix of a free-form Sudoku graph is
governed by the rules of Sudoku:
\begin{observation}
Two vertices of a free-form Sudoku graph are adjacent if and only if one of
the following mutually exclusive cases apply:
\begin{itemize}
\item[(B)] The associated cells belong to the same block of the tiling.
\item[(H)] The associated cells belong to the same row of the puzzle, but not to the same block.
\item[(V)] The associated cells belong to the same column of the puzzle, but not to the same block.
\end{itemize}
\end{observation}

The chosen partition into mutually exclusive cases immediately gives rise to a decomposition
of the adjacency matrix according to these three cases:
\begin{equation}\label{eq:decomp1}
A=L_B + L_H + L_V \in \RR^{m^2\times m^2}
\end{equation}

In this sense, the graph $\FSud(m, T)$ can be interpreted as a composition of three layers
of egdes, according to the cases (B), (H), (V). These layers can be viewed and studied as 
graphs of their own right.

\begin{proposition}\label{prop:lhlvmultipart}
For $i,j\in\{1,\ldots,m\}$ let $p_{ij}$ denote the number of cells in the $i$-th row (resp.\ column) of the puzzle
that belong to the $j$-th block of the partition. Then $L_H$ (resp.\ $L_V$) is an adjacency matrix
of the union of $m$ disjoint complete multipartite graphs $K_{p_{i1},\ldots,p_{im}}$, $i=1,\ldots,m$.
\end{proposition}

\begin{proof}
Consider a fixed row $i$ of the given Sudoku puzzle. We group the cells of the
considered row according to their respective block membership. 
In view of rule (H) we see that the vertices corresponding to the grouped cells
are adjacent if and only if they belong to different blocks.
So the groups of cells form $m$ independent sets of the respective sizes $p_{i1},\ldots,p_{im}$, further
all edges exist between different groups. The reasoning for  $L_V$ is analogous,
just consider a fixed column and rule (V).
\end{proof}

\begin{proposition}\label{prop:lbcomplete}
$L_B$ is an adjacency matrix of the union of $m$ disjoint complete graphs $K_m$.
\end{proposition}

\begin{proof}
Group the vertices according to block membership of their associated cells.
Since each block contains $m$ cells the result follows from rule (B).
\end{proof}

\begin{proposition}\label{prop:regcommute}
The following statements are equivalent:
\begin{enumerate}
\item[(a)] $L_H$ (resp.\ $L_V$) represents a regular graph.
\item[(b)] $L_H$ (resp.\ $L_V$) has constant row sum.
\item[(c)] $L_B$ commutes with $L_H$ (resp.\ $L_V$).
\end{enumerate}
\end{proposition}

\begin{proof}
Statements (a) and (b) are clearly equivalent. Next we show that statement (b) implies (c).
To this end, note that $(L_H)_{is}(L_B)_{sj}=1$ if and only if 
$(L_H)_{is}=(L_B)_{sj}=1$. This translates to the requirement that
the cells associated with vertices $i$ and $s$ are in the same row (but not the same block) and that
the cells associated with vertices $s$ and $j$ belong to same block.
Therefore, $(L_HL_B)_{ij}$ counts the vertices $s$ that meet the mentioned
requirement for given $i$, $j$.
In a similar manner we see that $(L_BL_H)_{ij}$ counts the vertices $s$ 
such that 
the cells associated with vertices $i$ and $s$ belong to same block and that
the cells associated with vertices $s$ and $j$ are in the same row (but do not belong to the same block).
If, however, $L_H$ has constant row sum, then it follows by rule (H) that
each cell has the same common number of cells that are in the same row but not in the same block as the considered cell.
Equivalently, since all rows have the same number of cells, each cell has the same common number $k$ of cells 
that are in the same row and in the same block as the considered cell. Therefore,
any block that has cells in a row must have exactly $k$ cells in that row. 
But then the two counts mentioned before are identical: 
$(L_HL_B)_{ij}=(L_BL_H)_{ij}$.

Conversely, assume that statement (c) is true. Consider two vertices $i$ and $j$ such that
their associated cells belong to different rows and different blocks.
Then $(L_BL_H)_{ij}=(L_HL_B)_{ij}$ implies that these two blocks have the same number of cells in the two rows.
Hence there exists a common number $k$ such that
any block that has cells in a row must have exactly $k$ cells in that row. 
Considering a single row of $L_H$, it follows that it has row sum $m-k$ since the puzzle row of
the cell associated with that matrix row contains $m$ cells and among them $k$ cells
that belong to the same block as the considered cell.
\end{proof}

Considering the third statement of Proposition \ref{prop:regcommute}, note that 
$L_H$ and $L_V$ need not commute even when both $L_H$ and $L_V$ have constant row sum,
cf.\ the graph $\FSud(4,T)$ for $T=\{\{1,8,9,16\}, \{3,6,10,15\}, \{2,7,11,14\}, \{4,5,12,13\}\}$ 
(here we number the cells from left to right, row after row).

\begin{theorem}[see \cite{CD}]
The characteristic polynomial of the complete multipartite graph $G=K_{p_1,\ldots,p_k}$
is
\[
\chi(G) = x^{n-k}\left( x^k - \sum\limits_{m=2}^k(m-1)\sigma_m x^{k-m}\right),
\]
where $\sigma_1 = \sum_{1\leq i \leq k} p_i$, $\sigma_2 = \sum_{1\leq i < j \leq k} p_ip_j$,
$\sigma_3 = \sum_{1\leq i<j<l\leq k}p_i p_j p_l$ and so on up to $\sigma_k = \prod_{1\leq i \leq k} p_i$.
\end{theorem}

\begin{corollary}\label{cor:multintegral}
Let $G=K_{q,\ldots,q}$ be a complete $k$-partite graph. Then its spectrum is
\[
\sigma(G)=\left\{  0^{(kq-k)}, (k-1)q^{(1)}, -q^{(k-1)} \right\}.
\]
\end{corollary}

\begin{proof}
	\begin{align*}
		\chi(G) &= x^{kq-k} \left( x^k - \binom{k}{2} q^2 x^{k-2} - 2q^3 \binom{k}{3} x^{k-3} - 3 q^4 \binom{k}{4} x^{k-4} - \ldots +(k-1) q^k\right)\\
		&= x^{kq-k} \left( x-(k-1)q \right) \left(x^{k-1} + \binom{k-1}{1} x^{k-2} q + \binom{k-1}{2} x^{k-3} q^2 + \ldots + q^{k-1}\right)\\
		&= x^{kq-k}(x-(k-1)q) (x+q)^{k-1}
	\end{align*}
\end{proof}

\begin{theorem}\label{thm:fsudint}
Given a free-form Sudoku puzzle, assume that the following conditions are met:
\begin{itemize}
\item[(i)] There exists a common number $q$ such that for each cell 
there exist exactly $q$ cells belonging to the same row and block as the considered cell (including itself).
\item[(ii)] A similar condition holds with `row' replaced by `column'. 
\item[(iii)] For any two cells $C_1$, $C_2$ belonging to different rows and columns,
the unique cell lying in the same row as $C_1$ and the same column as $C_2$ belongs to the same
block as $C_1$ if and only if exactly the same holds for 
the unique cell lying in the same column as $C_1$ and the same row as $C_2$.
\end{itemize}
Then the associated graph $\FSud(m, T)$ is integral.
\end{theorem}

\begin{proof}
Under condition (i) we have $p_{ij}\in\{0,q\}$ for all $i,j\in\{1,\ldots,m\}$ in
Proposition \ref{prop:lhlvmultipart}. Therefore, $L_H$ represents a union of
complete multipartite graphs $K_{q,\ldots,q}$. Hence $L_H$ is integral by Corollary \ref{cor:multintegral}.
Likewise, condition (ii) ensures that $L_V$ is integral.
$L_B$ is integral by \ref{prop:lbcomplete} and Corollary \ref{cor:multintegral} (for $r=1$).

We can conclude from Proposition \ref{prop:regcommute} and conditions (i) and (ii) 
that the pairs $L_H, L_B$ and $L_V, L_B$ commute.
Condition (iii) is equivalent to the condition that $L_H$ and $L_V$ commute.
Hence it follows that $L_H$, $L_V$, $L_B$ are simultaneously diagonizable.
Thus their sum $A=L_H+L_V+L_B$ is integral as well.
\end{proof}

It is easily checked that the classical $n$-Sudokus fulfill the conditions of Theorem \ref{thm:fsudint}.
Hence the theorem provides a new way of proving integrality of classical Sudoku graphs (e.g.\ different from the proofs in
\cite{semester}, \cite{sudejc}). Further, Theorem \ref{thm:fsudint} helps us identify many more
integral free-form Sudoku graphs besides the classical variants.

%-------------------------------------------------------------------------------------------------------
\section{Transformations}
%-------------------------------------------------------------------------------------------------------

In this section we consider a special kind of transformation of Sudokus puzzles and study its spectral properties.
To this end, we define the $k$-fold blow up of a free-form Sudoku. This is formed by replacing each cell of the original
Sudoku by a $k\times k$ arrangement of cells. The block partition $T'$ of the new graph is
derived from the original partition $T$ as follows. For each block $B\in T$ we create a block
$B'\in T'$ by collecting the replacement cells of all the cells in $B$.
In terms of graphs we see that the $k$-fold blow up transforms $\FSud(n,T)$ into $\FSud(kn,T')$.
Let us denote the latter graph by $\FSudUp{k}(n,T)$.
In Figure \ref{figshidoku3blowup} the $3$-fold blow up of a free-form Shidoku is illustrated. 
It also sketches the neighborhood of an exemplary vertex of the blown up graph.

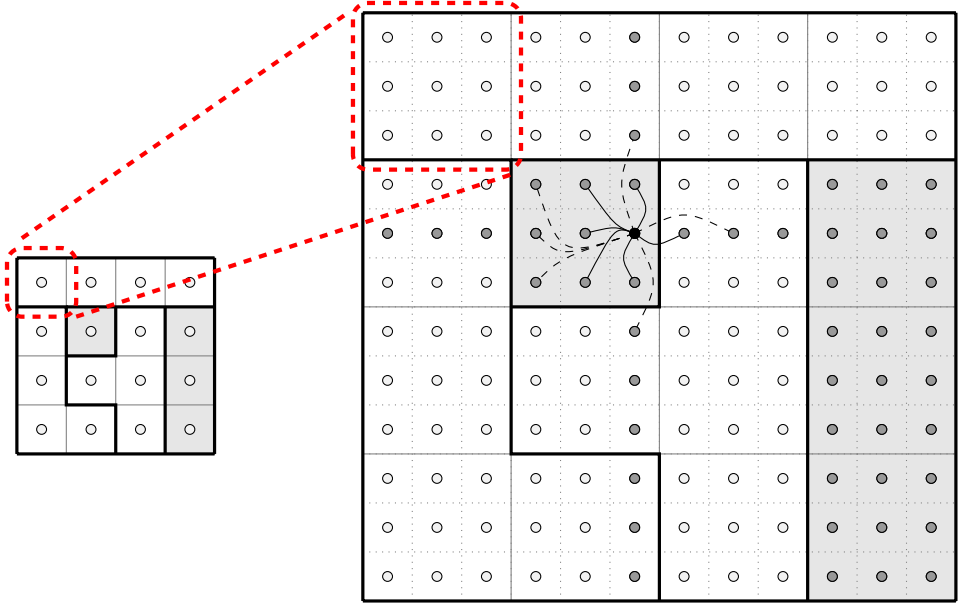
\begin{figure}%[ht]
\begin{center}
\tikzset{
	littleone/.pic={
		\draw[fill=gray!20]    (1,2) -- ++(1,0) -- ++(0,1) -- ++(-1,0); % shading first
		\draw[fill=gray!20]    (3,0) -- ++(1,0) -- ++(0,3) -- ++(-1,0); % shading first
		\draw[step=1cm, gray] (0,0) grid (4,4);
		\draw[step=4cm, very thick] (0,0) grid (4,4);
		\draw[very thick] (2,0) -- (2,1) -- (1,1) -- (1,3) -- (0,3);
		\draw[very thick] (1,3) -- (4,3);
		\draw[very thick] (3,0) -- (3,3);
		\draw[very thick] (1,2) -- (2,2) -- (2,3);
		\foreach \y in {0,...,3} \foreach \x in {0,...,3} \draw[fill=gray!10] (\x.5,\y.5) circle (0.1cm);
		\draw[ultra thick, red, dashed, rounded corners=2mm] (-0.2,2.8) rectangle ++(1.4,1.4);
	}
}
\begin{tikzpicture}[scale=0.65]
\begin{scope}[shift={(-7,3)}, scale=1, every node/.append style={transform shape}] % force scaling according to global scale
\pic{littleone};
\end{scope}
%---
\draw[fill=gray!20]    (3,6) -- ++(3,0) -- ++(0,3) -- ++(-3,0); % shading first
\draw[fill=gray!20]    (9,0) -- ++(3,0) -- ++(0,9) -- ++(-3,0); % shading first
\draw[step=1cm, gray, dotted] (0,0) grid (12,12);
\draw[step=3cm, gray] (0,0) grid (12,12);
\draw[step=12cm, very thick] (0,0) grid (12,12);
\draw[very thick] (6,0) -- (6,3) -- (3,3) -- (3,9) -- (0,9);
\draw[very thick] (3,9) -- (12,9);
\draw[very thick] (9,0) -- (9,9);
\draw[very thick] (3,6) -- (6,6) -- (6,9);
\draw (5.5,7.5) .. controls ++(0.3,0.5) .. ++(0,1); 	% neighbor above
\draw[dashed] (5.5,7.5) .. controls ++(-0.35,1.0) .. ++(0,2); 	% neighbor 2 above 
\draw (5.5,7.5) .. controls ++(0.4,-0.3) .. ++(1,0); 	% neighbor right
\draw[dashed] (5.5,7.5) .. controls ++(1.0,0.5) .. ++(2,0); 	% neighbor 2 right
\draw (5.5,7.5) .. controls ++(-0.3,-0.5) .. ++(0,-1); 	% neighbor below
\draw[dashed] (5.5,7.5) .. controls ++(0.5,-1.0) .. ++(0,-2); 	% neighbor 2 below 
\draw (5.5,7.5) .. controls ++(-0.4,+0.15) .. ++(-1,0); 	% neighbor left
\draw (5.5,7.5) .. controls ++(-0.4,+0.15) .. ++(-1,-1); 	% neighbor left below
\draw (5.5,7.5) .. controls ++(-0.4,+0.15) .. ++(-1,1); 	% neighbor left above
\draw[dashed] (5.5,7.5) .. controls ++(-1.5,-0.5) .. ++(-2,0); 	% neighbor 2 left
\draw[dashed] (5.5,7.5) .. controls ++(-1.5,-0.5) .. ++(-2,-1); 	% neighbor 2 left 1 below
\draw[dashed] (5.5,7.5) .. controls ++(-1.5,-0.5) .. ++(-2,1); 	% neighbor 2 left 1 above
% the vertices in general
\foreach \y in {0,...,11} \foreach \x in {0,...,11} \draw[fill=gray!10] (\x.5,\y.5) circle (0.1cm);
% neighbors of considered vertex
\foreach \y in {6,...,8} \foreach \x in {3,...,5} \draw[fill=gray!80] (\x.5,\y.5) circle (0.1cm); % block part 1
\foreach \y in {0,...,8} \foreach \x in {9,...,11} \draw[fill=gray!80] (\x.5,\y.5) circle (0.1cm);% block part 2
\foreach \y in {7} \foreach \x in {0,...,11} \draw[fill=gray!80] (\x.5,\y.5) circle (0.1cm);% row
\foreach \x in {5} \foreach \y in {0,...,11} \draw[fill=gray!80] (\x.5,\y.5) circle (0.1cm);% column
% considered vertex
\draw[fill] (5.5,7.5) circle (0.1cm);
\draw[ultra thick, red, dashed, rounded corners=2mm] (-0.2,8.8) rectangle ++(3.4,3.4);
\draw[ultra thick, red, dashed] (-6.9,7.4) -- ++(6.6,4.6);
\draw[ultra thick, red, dashed] (-5.8,5.8) -- ++(8.8,2.9);
\end{tikzpicture}
\end{center}
\caption{$3$-fold blow up of the free-form Shidoku puzzle from Figure \ref{figshidokuconstr} (b)}\label{figshidoku3blowup}
\end{figure}

Let us now investigate how the adjacency matrix of a blown up Sudoku graph can be determined from the
adjacency matrix of the original graph. To this end, we a assume that a fixed but otherwise arbitrary
numbering of the cells of the original free-form Sudoku is given. The vertices of the associated graph
shall be numbered accordingly. Further, we number the cells of a $k$-fold blow up of the given puzzle
as follows. Let $S_i$ denote the $k\times k$ subsquare of the blown up puzzle containing exactly the replacement cells of the original cell number $i$.
We number the $k^2n^2$ cells by subsequently numbering all vertices in $S_1$, then $S_2$ and so on.
Inside each $S_i$ we number the cells by starting in the top left corner and proceeding from left to right, advancing row by row.

Next, note the following facts about the blow up operation:

\begin{observation}\leavevmode\vspace{-0.5\baselineskip}
\begin{itemize}
\item If two cells $i$ and $j$ lie in the same row (resp.\ column) of the original puzzle, then
all cells sharing the same relative row (resp.\ column) index inside $S_i$ and/or $S_j$ lie in the same row (resp.\ column) of the blown up puzzle.
\item If two cells $i$ and $j$ lie in the same block of the original puzzle, then
all cells from $S_i$ and $S_j$ lie in the same block of the blown up puzzle.
\end{itemize}
\end{observation}

In view of these facts and due to the chosen cell numbering of the blown up puzzle
we can construct the adjacency matrix $\Up{A}$ of $\FSudUp{k}(n,T)$ from
the adjacency matrix $A=(a_{ij})$ of $\FSud(n,T)$ by replacing each entry $a_{ij}$ of $A$ by a $(k^2\times k^2)$-matrix
that solely depends on whether (in the original puzzle) cell number $j$ is in the same block as cell number $i$, 
or otherwise in the same row or column as $j$ 
or lies somewhere else. 

For this purpose we define the symbolic template adjacency matrix $\TM(A)=(t_{ij})$ as follows:
\begin{itemize}
\item $t_{ii} = \mathrm{D}$,
\item $t_{ij} = \mathrm{B}$ if $i\not=j$ and cells $i,j$ are in the same block of the original puzzle,
\item $t_{ij} = \mathrm{H}$ if cells $i,j$ are not in the same block but in the same row of the original puzzle,
\item $t_{ij} = \mathrm{V}$ if cells $i,j$ are not in the same block but in the same column of the original puzzle,
\item $t_{ij} = \mathrm{N}$ else.
\end{itemize}

The next step is to define the matrices that will be substituted into the template matrix.
But first we need some building blocks. Let $I_r$ denote the identity matrix, $J_r$ the all ones matrix and $N_r$ the zero matrix of size $r\times r$.
Further, we will make use of the Kronecker product $\otimes$ of real matrices, cf.\ \cite{HJ}:

\begin{definition}\label{def:kron}
Let $A\in\RR^{p\times q}$ and $B\in\RR^{r\times s}$. Then the Kronecker product of $A$ and $B$ is defined as the block matrix
\[
	A\otimes B = \left(
\begin{array}{ccc}
 a_{11}B & \dots a_{1q} B \\
 \vdots  & \vdots \\
a_{p1}B & \dots a_{pq} B 
 \end{array}
\right) \in\RR^{pr\times sq}.
\]
\end{definition}

Fixing the given blow up factor $k$, we define the following $(k^2\times k^2)$-matrices:
\begin{equation}\label{matHVBN}
H = I_k \otimes J_k, ~~V = J_k \otimes I_k, ~~B = J_{k^2}, ~~ D = J_{k^2} - I_{k^2}, ~~ N = N_{k^2}.
\end{equation}

The following result should now be self-evident:
\begin{proposition}
Given an $n\times n$ free-form Sudoku puzzle with block partition $T$, 
the adjacency matrix of $\FSudUp{k}(n,T)$ (with respect to the vertex numbering mentioned earlier) can be obtained
from the symbolic template matrix $\TM(A)=(t_{ij})$
of the adjacency matrix $A=(a_{ij})$ of the graph $\FSud(n,T)$
by replacing each entry $t_{ij}$ of $\TM(A)$
by the contents of the matching matrix from \eqref{matHVBN} that has the same
name as the symbolic value of $t_{ij}$ suggests.
\end{proposition}

The replacement process can be expressed by means of the Kronecker product.
Let us take the template matrix $\TM(A)$ and use it to partition
the non-zero entries $a_{ij}$ of $A$ according to their respective symbols $t_{ij}$.
Setting $L_D=I_{n^2}$, the blown up adjacency matrix $\Up{A}$ can now be expressed as follows:

\begin{proposition}
\begin{equation}\label{eq:decomp2}
\Up{A} = L_B \otimes B + L_H \otimes H + L_V \otimes V + L_D \otimes D.
% L_B \otimes B + L_H \otimes H + L_V \otimes V + I_{n^2} \otimes J_{k^2} - I_{k^2n^2}
\end{equation}
\end{proposition}

We demonstrate the process of deriving the adjacency matrix of the blown up graph by an illustrative example. Consider the free-form Shidoku depicted in
Figure \ref{figshidokuconstr} (b) and, for the sake of simplicity, number the cells successively by starting in the top left corner
and proceeding from left to right in each row, advancing row by row. 
Considering the associated graph $\FSud(2,T)$ with tiling $T=\{\{1,2,3,4\},\ab \{5,9,13,14\},\ab \{6,8,12,16\},\ab \{7,10,11,15\}\}$,
this numbering yields the following adjacency matrix $A$ and symbolic template matrix $\TM$:

\begin{center}
\resizebox{0.95\hsize}{!}{
$ A=
\left(
{\begin{array}{rrrrrrrrrrrrrrrr}
0 & 1 & 1 & 1 & 1 & 0 & 0 & 0 & 1 & 0 & 0 & 0 & 1 & 0 & 0 & 0 \\
1 & 0 & 1 & 1 & 0 & 1 & 0 & 0 & 0 & 1 & 0 & 0 & 0 & 1 & 0 & 0 \\
1 & 1 & 0 & 1 & 0 & 0 & 1 & 0 & 0 & 0 & 1 & 0 & 0 & 0 & 1 & 0 \\
1 & 1 & 1 & 0 & 0 & 0 & 0 & 1 & 0 & 0 & 0 & 1 & 0 & 0 & 0 & 1 \\
1 & 0 & 0 & 0 & 0 & 1 & 1 & 1 & 1 & 0 & 0 & 0 & 1 & 1 & 0 & 0 \\
0 & 1 & 0 & 0 & 1 & 0 & 1 & 1 & 0 & 1 & 0 & 1 & 0 & 1 & 0 & 1 \\
0 & 0 & 1 & 0 & 1 & 1 & 0 & 1 & 0 & 1 & 1 & 0 & 0 & 0 & 1 & 0 \\
0 & 0 & 0 & 1 & 1 & 1 & 1 & 0 & 0 & 0 & 0 & 1 & 0 & 0 & 0 & 1 \\
1 & 0 & 0 & 0 & 1 & 0 & 0 & 0 & 0 & 1 & 1 & 1 & 1 & 1 & 0 & 0 \\
0 & 1 & 0 & 0 & 0 & 1 & 1 & 0 & 1 & 0 & 1 & 1 & 0 & 1 & 1 & 0 \\
0 & 0 & 1 & 0 & 0 & 0 & 1 & 0 & 1 & 1 & 0 & 1 & 0 & 0 & 1 & 0 \\
0 & 0 & 0 & 1 & 0 & 1 & 0 & 1 & 1 & 1 & 1 & 0 & 0 & 0 & 0 & 1 \\
1 & 0 & 0 & 0 & 1 & 0 & 0 & 0 & 1 & 0 & 0 & 0 & 0 & 1 & 1 & 1 \\
0 & 1 & 0 & 0 & 1 & 1 & 0 & 0 & 1 & 1 & 0 & 0 & 1 & 0 & 1 & 1 \\
0 & 0 & 1 & 0 & 0 & 0 & 1 & 0 & 0 & 1 & 1 & 0 & 1 & 1 & 0 & 1 \\
0 & 0 & 0 & 1 & 0 & 1 & 0 & 1 & 0 & 0 & 0 & 1 & 1 & 1 & 1 & 0
\end{array}}
 \right),~
\TM(A) = \left(
{ 
\newcommand{\bB}{\text{\bf B}}
\newcommand{\bH}{\text{\bf H}}
\newcommand{\bV}{\text{\bf V}}
\newcommand{\bD}{\text{\itshape D}}
\newcommand{\bN}{\text{ N}}
\begin{array}{cccccccccccccccc}
\bD & \bB & \bB & \bB & \bV & \bN & \bN & \bN & \bV & \bN & \bN & \bN & \bV & \bN & \bN & \bN \\
\bB & \bD & \bB & \bB & \bN & \bV & \bN & \bN & \bN & \bV & \bN & \bN & \bN & \bV & \bN & \bN \\
\bB & \bB & \bD & \bB & \bN & \bN & \bV & \bN & \bN & \bN & \bV & \bN & \bN & \bN & \bV & \bN \\
\bB & \bB & \bB & \bD & \bN & \bN & \bN & \bV & \bN & \bN & \bN & \bV & \bN & \bN & \bN & \bV \\
\bV & \bN & \bN & \bN & \bD & \bH & \bH & \bH & \bB & \bN & \bN & \bN & \bB & \bB & \bN & \bN \\
\bN & \bV & \bN & \bN & \bH & \bD & \bH & \bB & \bN & \bV & \bN & \bB & \bN & \bV & \bN & \bB \\
\bN & \bN & \bV & \bN & \bH & \bH & \bD & \bH & \bN & \bB & \bB & \bN & \bN & \bN & \bB & \bN \\
\bN & \bN & \bN & \bV & \bH & \bB & \bH & \bD & \bN & \bN & \bN & \bB & \bN & \bN & \bN & \bB \\
\bV & \bN & \bN & \bN & \bB & \bN & \bN & \bN & \bD & \bH & \bH & \bH & \bB & \bB & \bN & \bN \\
\bN & \bV & \bN & \bN & \bN & \bV & \bB & \bN & \bH & \bD & \bB & \bH & \bN & \bV & \bB & \bN \\
\bN & \bN & \bV & \bN & \bN & \bN & \bB & \bN & \bH & \bB & \bD & \bH & \bN & \bN & \bB & \bN \\
\bN & \bN & \bN & \bV & \bN & \bB & \bN & \bB & \bH & \bH & \bH & \bD & \bN & \bN & \bN & \bB \\
\bV & \bN & \bN & \bN & \bB & \bN & \bN & \bN & \bB & \bN & \bN & \bN & \bD & \bB & \bH & \bH \\
\bN & \bV & \bN & \bN & \bB & \bV & \bN & \bN & \bB & \bV & \bN & \bN & \bB & \bD & \bH & \bH \\
\bN & \bN & \bV & \bN & \bN & \bN & \bB & \bN & \bN & \bB & \bB & \bN & \bH & \bH & \bD & \bH \\
\bN & \bN & \bN & \bV & \bN & \bB & \bN & \bB & \bN & \bN & \bN & \bB & \bH & \bH & \bH & \bD
\end{array}}
 \right) $
}
\end{center}

Next construct the $3$-fold blowup as indicated in Figure \ref{figshidoku3blowup}. We number its $9\cdot 16=144$
cells as described earlier. The adjacency matrix of $\FSudUp{3}(2,T)$
is now readily obtained by replacing each symbol in $\TM(A)$ by one of the following matrices with
the matching name:

\begin{center}
\resizebox{0.95\hsize}{!}{
$
H = \left( 
{\begin{array}{rrrrrrrrr}
1 & 1 & 1 & 0 & 0 & 0 & 0 & 0 & 0 \\
1 & 1 & 1 & 0 & 0 & 0 & 0 & 0 & 0 \\
1 & 1 & 1 & 0 & 0 & 0 & 0 & 0 & 0 \\
0 & 0 & 0 & 1 & 1 & 1 & 0 & 0 & 0 \\
0 & 0 & 0 & 1 & 1 & 1 & 0 & 0 & 0 \\
0 & 0 & 0 & 1 & 1 & 1 & 0 & 0 & 0 \\
0 & 0 & 0 & 0 & 0 & 0 & 1 & 1 & 1 \\
0 & 0 & 0 & 0 & 0 & 0 & 1 & 1 & 1 \\
0 & 0 & 0 & 0 & 0 & 0 & 1 & 1 & 1
\end{array}}
 \right),
V= \left(
{\begin{array}{rrrrrrrrr}
1 & 0 & 0 & 1 & 0 & 0 & 1 & 0 & 0 \\
0 & 1 & 0 & 0 & 1 & 0 & 0 & 1 & 0 \\
0 & 0 & 1 & 0 & 0 & 1 & 0 & 0 & 1 \\
1 & 0 & 0 & 1 & 0 & 0 & 1 & 0 & 0 \\
0 & 1 & 0 & 0 & 1 & 0 & 0 & 1 & 0 \\
0 & 0 & 1 & 0 & 0 & 1 & 0 & 0 & 1 \\
1 & 0 & 0 & 1 & 0 & 0 & 1 & 0 & 0 \\
0 & 1 & 0 & 0 & 1 & 0 & 0 & 1 & 0 \\
0 & 0 & 1 & 0 & 0 & 1 & 0 & 0 & 1
\end{array}}
 \right),$
} \\
$ B = J_9, D= J_9 - I_9, N = N_9.$
\end{center}

Our goal is to study the eigenvalues of blown up free-form Sudokus and express them in terms of the
eigenvalues of the original puzzle.

For the rest of this section, we consider an arbitrary but fixed free-form $n\times n$ Sudoku puzzle with tiling $T$.
Let $A$ be the adjacency matrix of its associated graph $\FSud(n,T)$ and let $\Up{A}$ be the
adjacency matrix of $\FSudUp{k}(n,T)$. We assume that $L_V$, $L_H$, $L_B$ etc.\ denote the matrices
appearing in equations \eqref{eq:decomp1} and \eqref{eq:decomp2}. Further, let $\Eig(M)$ represent the set of 
all eigenvectors of a given matrix (or even a graph) $M$ and let $\Ker(M)$ be the set of all eigenvectors of $M$ associated with eigenvalue $0$.
Note that neither set contains the null vector.

In the following, we will make use of the following properties of the Kronecker product:

\begin{theorem}[see \cite{HJ}]\leavevmode\vspace{-0.5\baselineskip}
	\begin{enumerate}
	\item $(\alpha A) \otimes B = A \otimes (\alpha B) = \alpha(A \otimes B) \textall \alpha \in \RR, A \in \RR^{p\times q},B \in \RR^{r\times s}$
	\item $(A \otimes B) \otimes C = A \otimes (B \otimes C)  \textall A \in \RR^{m\times n}, B \in \RR^{p\times q}, C \in \RR^{r\times s}$
	\item $(A + B) \otimes C = A \otimes C + B \otimes C  \textall A,B \in \RR^{p\times q}, C \in \RR^{r\times s}$
	\item $A \otimes (B+C) = A \otimes B + A \otimes C  \textall A \in \RR^{p\times q}, B,C \in \RR^{r\times s}$
	\item $(A  \otimes B) (C \otimes D) = AC \otimes BD  \textall A \in \RR^{p\times q}, B \in \RR^{r\times s}, C \in \RR^{q\times k}, D \in \RR^{s\times l}$
	\end{enumerate}
\end{theorem}

Note that the Kronecker
product formally also includes the case where one or both factors are vectors.

\begin{lemma} \label{l1}
Suppose that $x$ and $y$ are eigenvectors of $L_V$ and $J_k$ corresponding to the eigenvalues $\lambda$ and $\alpha$, respectively.
Then, for any $z \in\Ker(J_k)$, $x \otimes y \otimes z$ is an eigenvector of $\Up{A}$ corresponding to eigenvalue $\lambda\alpha -1$.
\end{lemma}

\begin{proof} 
We use equations \eqref{matHVBN}, \eqref{eq:decomp2}  and the facts $L_Vx=\lambda x$, $J_ky=\alpha y$, $J_kz=0$.
\begin{align*}
&\Up{A} (x\otimes y \otimes z)  \\
&~~= (L_B \otimes B + L_H \otimes H + L_V \otimes V + L_D \otimes D)(x\otimes y \otimes z)  \\
&~~= (L_B \otimes J_{k^2} + L_H \otimes I_k \otimes J_k + L_V \otimes J_k \otimes I_k \\
&\hspace*{0.2\linewidth}+ L_D \otimes J_{k^2} - L_D \otimes I_{k^2})(x\otimes y \otimes z)  \\
&~~= ((L_Bx) \otimes (J_{k}y) \otimes (J_{k}z)) + ((L_Hx) \otimes (I_ky) \otimes (J_kz))  \\
&\hspace*{0.2\linewidth} + ((L_Vx) \otimes (J_ky) \otimes (I_kz)) +((L_Dx) \otimes (J_{k}y) \otimes (J_{k}z)) \\
&\hspace*{0.2\linewidth} - ((L_Dx) \otimes (I_{k}y) \otimes (I_{k}z))  \\
&~~= (\lambda x)\otimes (\alpha y) \otimes z - x \otimes y \otimes z \\
&~~= (\lambda \alpha - 1) (x \otimes y \otimes z).
\end{align*}
\end{proof}

\begin{lemma} \label{l2}
Suppose that $x$ and $y$ are eigenvectors of $L_H$ and $J_n$ corresponding to the eigenvalues $\lambda$ and $\alpha$,  respectively. 
Then, for any $z \in\Ker(J_n)$, $x \otimes y \otimes z$ is an eigenvector of $\Up{A}$ corresponding to eigenvalue $\lambda\alpha -1$.
\end{lemma}

\begin{proof}
Similar to the proof of Lemma \ref{l1}.
\end{proof}

The previous two lemmas will play a role in the construction of a basis of eigenvectors of a blown up Sudoku graph.  
To this end, we need to know the intersection of the spans of the two vector sets mentioned there.
But first note the following obvious facts:
\begin{proposition}\label{pro:jk}\leavevmode\vspace{-0.5\baselineskip}
\begin{enumerate}
\item The spectrum of $J_k$ is $\{k^{(1)}, 0^{(k-1)}\}$.
\item The set
\begin{align*}
&\Kb_J := \{ (1, -1, 0, 0, \ldots, 0, 0)^T, (1 , 0 , -1 , 0 , \ldots , 0 , 0)^T, \\
&\hspace*{2cm}\ldots, (1 , 0 , 0 , 0 , \ldots , 0 , -1)^T \}
\end{align*}
is a maximal linearly independent subset of $\Ker(J_k)$.
\item The set $\{\jv_k\} \cup \Kb_J$ is a maximal linearly independent subset of $\Eig(J_k)$.
\item $\jv_k \perp \Ker(J_n).$
\end{enumerate}
\end{proposition}

For the next lemma we define the following sets:
\begin{align*}
X_H &= \{ x\otimes y\otimes z :~ x\in\Eig(L_H), y\in\Ker(J_k), z\in\Eig(J_k)\}, \\
X_V &= \{ x\otimes y\otimes z :~ x\in\Eig(L_V), y\in\Eig(J_k), z\in\Ker(J_k)\}, \\
\tilde X_H & = \{x\otimes y\otimes z :~ x\in\Eig(L_H),~ y,z\in\Ker(J_k)\}, \\
\tilde X_V & = \{x\otimes y\otimes z :~ x\in\Eig(L_V),~ y,z\in\Ker(J_k)\}.
\end{align*}

\begin{lemma} \label{l12int}
\begin{align*}
\Span(X_H)\cap\Span(X_V)  = \Span(\tilde X_H) = \Span(\tilde X_V).
\end{align*}
\end{lemma}

\begin{proof}
Since $\Ker(J_k)\subseteq \Eig(J_k)$ it is obvious that
\begin{align*}
\Span(\tilde X_H) \subseteq \Span(X_H)\cap\Span(X_V).
\end{align*}
Conversely, note that both $L_H$ and $L_V$ are symmetric and therefore diagonizable, i.e.\ $\Span(\Eig(L_H)) = \Span(\Eig(L_V)) = \RR^{n^2}$.
Using this we conclude
\begin{align*}
\Span(X_H) & \subseteq \RR^{n^2} \otimes \Span(\Ker(J_k)) \otimes \Span(\Eig(J_k)), \\
\Span(X_V) & \subseteq \RR^{n^2} \otimes \Span(\Eig(J_k)) \otimes \Span(\Ker(J_k)), \\
\Span(\tilde X_H) & = \RR^{n^2} \otimes \Span(\Ker(J_k)) \otimes \Span(\Ker(J_k)).
\end{align*}
Once again employing the fact that $\Ker(J_k)\subseteq \Eig(J_k)$, we arrive at
\begin{align*}
\Span(X_H) \cap \Span(X_V) \subseteq \Span(\tilde X_H).
\end{align*}
\end{proof}

\begin{lemma} \label{l3}
Let $x \in \mathbb{R}^{n^2}$ and $y,z \in \Ker(J_k)$. Then $x\otimes y \otimes z$ is an eigenvector of $\Up{A}$ corresponding to eigenvalue $-1 $.
\end{lemma}

\begin{proof} 
\begin{align*}
&\Up{A} (x\otimes y \otimes z)  \\
&~~= (L_B \otimes J_{k^2} + L_H \otimes I_k \otimes J_k + L_V \otimes J_k \otimes I_k \\
&\hspace*{0.2\linewidth}+ L_D \otimes J_{k^2} - L_D \otimes I_{k^2})(x\otimes y \otimes z)  \\
&~~= ((L_Bx) \otimes (J_{k}y) \otimes (J_{k}z)) + ((L_Hx) \otimes (I_ky) \otimes (J_kz))  \\
&\hspace*{0.2\linewidth} + ((L_Vx) \otimes (J_ky) \otimes (I_kz)) +((L_Dx) \otimes (J_{k}y) \otimes (J_{k}z)) \\
&\hspace*{0.2\linewidth} - ((L_Dx) \otimes (I_{k}y) \otimes (I_{k}z))  \\
&~~= - ((I_{n^2}x) \otimes (I_{k}y) \otimes (I_{k}z)) \\
&~~= - (x \otimes y \otimes z).
\end{align*}
\end{proof}

In the following, let $\jv_r$ denote the all ones vector of dimension $r$.

\begin{lemma} \label{l4}
For any eigenvector $x$ of the matrix $k^2 L_B + k L_H + k L_V$, the vector $x \otimes \jv_{k^2}$ is an eigenvector of $\Up{A}$ 
corresponding to eigenvalue $\lambda + k^2 -1$.
\end{lemma}

\begin{proof} 
\begin{align*}
&\Up{A} (x\otimes \jv_{k^2})   \\
&~~= (L_B \otimes B + L_H \otimes H + L_V \otimes V + L_D \otimes D)(x\otimes \jv_{k^2}) \\
&~~= ((L_Bx) \otimes (J_{k^2}\jv_{k^2})) + ((L_Hx) \otimes (H\jv_{k^2}))  + ((L_Vx) \otimes (V\jv_{k^2})) \\
&\hspace*{0.2\linewidth} +((L_Dx) \otimes (J_{k^2}\jv_{k^2})) - ((L_Dx) \otimes (I_{k^2}\jv_{k^2}))  \\
&~~= ((L_Bx) \otimes ({k^2}\jv_{k^2})) + ((L_Hx) \otimes (k\jv_{k^2}))  + ((L_Vx) \otimes (k\jv_{k^2})) \\
&\hspace*{0.2\linewidth} +((L_Dx) \otimes ({k^2}\jv_{k^2})) - ((L_Dx) \otimes \jv_{k^2}) \\
&~~= (({k^2}L_Bx) \otimes \jv_{k^2}) + ((kL_Hx) \otimes \jv_{k^2})  + ((kL_Vx) \otimes \jv_{k^2}) \\
&\hspace*{0.2\linewidth} +(({k^2}-1)(L_Dx) \otimes \jv_{k^2})  \\
&~~= (({k^2}L_B+kL_H+kL_V)x \otimes \jv_{k^2}) +(({k^2}-1)x \otimes \jv_{k^2}) \\
&~~= (\lambda + k^2 - 1) (x \otimes \jv_{k^2}). \\
\end{align*}
\end{proof}

Before we present the main result of this section we need one more technical lemma.

\begin{lemma} \label{l5}
Let  $\{ Y_i \}_{i=1}^n$ be a set of linearly independent vectors. Then, for any set $\{ X_i \}_{i=1}^m$ of nonzero vectors and 
any function \[\phi:\{1,2,3,\dots,n\} \to \{1,2,3,\dots,m \},\] the set $\{ X_{\phi(i)} \otimes Y_i \}_{i=1}^n$ is linearly independent.
\end{lemma}

\begin{proof}
Suppose otherwise that
\begin{align*}
\sum\limits_{i=1}^n c_i (X_{\phi(i)} \otimes Y_i) = 0
\end{align*}
for suitable numbers $c_i\in\RR$. Let $X_{\phi(i)}=(x_{1,\phi(i)},\ldots,x_{r,\phi(i)})^T$.
By the definition of the Kronecker product we have
\begin{align*}
\sum\limits_{i=1}^n c_i (X_{\phi(i)} \otimes Y_i) = \cvec{ \sum\limits_{i=1}^n c_i x_{1,\phi(i)} Y_i \\ \vdots \\ \sum\limits_{i=1}^n c_i x_{r,\phi(i)} Y_i}=0.
\end{align*}
Thus, for each $j$ we have
\begin{align*}
\sum\limits_{i=1}^n c_i x_{j,\phi(i)} Y_i=0.
\end{align*}
But due to the linear independence of the vectors $Y_i$ we see that
$c_i x_{j,\phi(i)} = 0$
for every $j=1,\ldots,r$ and $i=1,\ldots,n$. Now assume that $c_{i^\ast}\not=0$ for some index $i^\ast$. Then
$x_{j,\phi(i^\ast)}=0$ for all $j$, therefore $X_{\phi(i^\ast)}=0$.
But this is impossible since $\{ X_i \}_{i=1}^m$ is a set of nonzero vectors.
\end{proof}

For what follows, let $\Bb_V$, $\Bb_H$ and $\Bb_M$ denote arbitrary maximal linearly independent subsets of $\Eig(L_V)$, $\Eig(L_H)$
and $\Eig(k^2 L_B + k L_V + k L_H)$, respectively.
Further, let $\Eb=\{e_1,\ldots,e_{n^2}\}$ be the standard basis of $\RR^{n^2}$, where $e_i$ denotes the $i$-th
unit vector.

\begin{theorem}\label{th:main}
Given a graph $\FSud(n,T)$, define the sets
\begin{align*}
 \Xb_V &= \{ x \otimes \jv_k \otimes y :~ x \in \Bb_V, y \in \Kb_J\} ,\\
 \Xb_H &= \{ x \otimes y \otimes \jv_k :~ x \in \Bb_H, y\in \Kb_J\} , \\
 \Xb_E &= \{ x \otimes y \otimes z :~ x\in \Eb,~ y,z \in \Kb_J\}, \\
 \Xb_M &= \{ x \otimes \jv_{k^2} :~ x \in \Bb_M\}.
\end{align*}
Then, their union $\Xb_V\cup\Xb_H\cup\Xb_E\cup\Xb_M$ forms a maximal linearly independent subset of $\Eig(\FSudUp{k}(n,T))$.
\end{theorem}

\begin{proof}
By construction and Lemma \ref{l5}, each of the sets $\Xb_V$, $\Xb_H$, $\Xb_E$, $\Xb_M$ in itself
is linearly independent. Further, by construction and Proposition \ref{pro:jk}, the spans of
these four sets are mutually disjoint (neglecting the null vector). Moreover, Lemmas \ref{l1}, \ref{l2}, \ref{l3} and \ref{l5} 
guarantee that the union contains only eigenvectors of $\FSudUp{k}(n,T)$.
Finally, note that
\begin{align*}
\vert \Bb_V \vert = \vert \Bb_H \vert = \vert \Eb \vert  = n^2,~~ \vert \Kb_j \vert = k-1
\end{align*}
so that
\begin{align*}
\vert \Xb_V \vert = \vert \Xb_H \vert = (k-1)n^2,~~ \vert \Xb_E \vert = (k-1)^2n^2,~~ \vert \Xb_M \vert = n^2
\end{align*}
and thus 
\begin{align*}
\vert \Xb_V \vert + \vert \Xb_H \vert + \vert \Xb_E \vert +  \vert \Xb_M \vert = 2(k-1)n^2 + (k-1)^2n^2 + n^2 = k^2n^2.
\end{align*}
\end{proof}

Looking closer at Theorem \ref{th:main}, we see that if $L_V$, $L_H$, $k^2L_B+kL_H+kL_V$ were all integral, 
then $\FSudUp{k}(n,T)$ would be integral as well.

\begin{corollary}
Under the conditions stated in Theorem \ref{thm:fsudint}
it follows that the blown up graph $\FSudUp{k}(n,T)$ is integral for every $k$.
\end{corollary}

\begin{proof}
From the proof of Theorem \ref{thm:fsudint} it follows that
$L_V$, $L_H$ and $L_B$ are all integral and commute with each other,
hence they are simultaneously diagonizable. Consequently, $k^2L_B+kL_H+kL_V$ is integral
and therefore $\FSudUp{k}(n,T)$.
\end{proof}

Owing to Theorem \ref{th:main}, we can use Lemmas \ref{l1}, \ref{l2}, \ref{l3} and \ref{l5} to establish
the spectrum of a $k$-fold blow up from its original. 
Interestingly, we can explicitly predict that the largest eigenvalue stems from the set $\Xb_M$:

\begin{theorem}
Given a graph $\FSud(n,T)$, let $\lambda$ be the largest eigenvalue of the associated matrix $k^2 L_B + k L_V + k L_H$. Then
$\lambda + k^2-1$ is the largest eigenvalue of $\FSudUp{k}(n,T)$.
\end{theorem}

\begin{proof}
For the purposes of this proof we renumber  the vertices of $\FSud(n,T)$ such that we sequentially number the vertices with one block, then continue
with the next block and so on. The order in which the vertices are numbered with a single block is arbitrary.
With respect to this vertex order the matrix $L_B$ assumes the form $I_n\otimes J_n - I_{n^2}$.
Clearly, this matrix contains $J_n-I_n$ as a principal submatrix. Consequently,
the matrix $M:=k^2 L_B + k L_V + k L_H$ contains the matrix $k^2J_n-k^2I_n$ as a principal submatrix,
the latter having maximum eigenvalue $k^2(n-1)$. By virtue of eigenvalue interlacing (see e.g.\ \cite{spectra})
we conclude that $\lambda > (n-1)k^2$.
So, according to Lemma \ref{l4}, the largest eigenvalue of $\FSudUp{k}(n,T)$ originating from the set $\Xb_M$ is at least
$(n-1)k^2 + k^2 -1 = nk^2-1$.
We will now show that the largest eigenvalues originating from $\Xb_V$, $\Xb_H$ and $\Xb_B$ are smaller than this bound.

Since the largest eigenvalue of a matrix is bounded from above by the maximum row sum of the matrix it is clear that
the maximum eigenvalue both of $L_H$ and $L_V$ can be at most $n-1$.
Now recall from Proposition \ref{pro:jk} that $k$ is the maximum eigenvalue of $J_k$.
Combining these findings, it now follows from Lemmas \ref{l1} and \ref{l2} that 
none of the eigenvalues associated with the vectors of the sets $\Xb_V$ and $\Xb_H$ exceeds $(n-1)k-1$,
which is less than the lower bound given for $\Xb_M$.
Finally, Lemma \ref{l3} tells us that no positive eigenvalue of  $\FSudUp{k}(n,T)$ originates from $\Xb_B$.
\end{proof}

%-------------------------------------------------------------------------------------------------------
\section{Conclusion}
%-------------------------------------------------------------------------------------------------------

Up to now, it seems that free-form Sudokus have not been researched at all.
Providing a starting point, we have studied integrality of these graphs. Moreover,
we have presented the blow up operation and 
shown how to obtain the eigenvalues of blown up free-form Sudokus from their originals.
We would like to 
inspire more research on this topic, in particular regarding further spectral properties of free-form Sudoku graphs.
Let us therefore close with the following open questions:
\begin{enumerate}
\item Can we find a precise condition on the tiling that allows us to predict when exactly a
free-form Sudoku graph is integral or not?
\item If a given Sudoku is integral, is its blown up version always integral as well?
\item Can a blown up free-form Sudoku be integral although its original Sudoku is not?
\end{enumerate}

%-------------------------------------------------------------------------------------------------------

%\nocite{*}                          
\bibliographystyle{acm}            
\bibliography{freeform_sudoku}  

\begin{thebibliography}{10}

\bibitem{zbMATH06475256}
{\sc {Arnold}, E., {Field}, R., {Lorch}, J., {Lucas}, S., and {Taalman}, L.}
\newblock {Nest graphs and minimal complete symmetry groups for magic Sudoku
  variants.}
\newblock {\em {Rocky Mt. J. Math.} 45}, 3 (2015), 887--901.

\bibitem{zbMATH06052119}
{\sc {Cameron}, P.~J., {Hilton}, A. J.~W., and {Vaughan}, E.~R.}
\newblock {An analogue of Ryser's theorem for partial Sudoku squares.}
\newblock {\em {J. Comb. Math. Comb. Comput.} 80\/} (2012), 47--69.

\bibitem{zbMATH06759817}
{\sc {Chen}, H., and {Cooper}, C.}
\newblock {Solving sudoku: structures and strategies.}
\newblock {\em {Missouri J. Math. Sci.} 29}, 1 (2017), 12--18.

\bibitem{zbMATH06246071}
{\sc {Cooper}, J., and {Kirkpatrick}, A.}
\newblock {Critical sets for Sudoku and general graph colorings.}
\newblock {\em {Discrete Math.} 315-316\/} (2014), 112--119.

\bibitem{spectra}
{\sc {Cvetkovi\'c}, D.~M., {Doob}, M., and {Sachs}, H.}
\newblock {\em {Spectra of graphs. Theory and applications. 3rd rev. a. enl.
  ed.}}, 3rd rev. a. enl. ed.~ed.
\newblock Leipzig: J. A. Barth Verlag, 1995.

\bibitem{zbMATH03908461}
{\sc {Cvetkovi\'c}, D.~M., and {Petri\'c}, M.~V.}
\newblock {Connectedness of the non-complete extended p-sum of graphs.}
\newblock {\em {Zb. Rad., Prir.-Mat. Fak., Univ. Novom Sadu, Ser. Mat.} 13\/}
  (1983), 345--352.

\bibitem{CD}
{\sc {Delorme}, C.}
\newblock {Eigenvalues of complete multipartite graphs.}
\newblock {\em {Discrete Math.} 312}, 17 (2012), 2532--2535.

\bibitem{zbMATH06142305}
{\sc {Deng}, X.~Q., and {Da Li}, Y.}
\newblock {A novel hybrid genetic algorithm for solving Sudoku puzzles.}
\newblock {\em {Optim. Lett.} 7}, 2 (2013), 241--257.

\bibitem{zbMATH06751283}
{\sc {D'haeseleer}, J., {Metsch}, K., {Storme}, L., and {Van de Voorde}, G.}
\newblock {On the maximality of a set of mutually orthogonal sudoku Latin
  squares.}
\newblock {\em {Des. Codes Cryptography} 84}, 1-2 (2017), 143--152.

\bibitem{EM}
{\sc {Elsholtz}, C., and {M\"utze}, A.}
\newblock {Sudoku im Mathematikunterricht.}
\newblock {\em {Math. Semesterber.} 54}, 1 (2007), 69--93.

\bibitem{zbMATH05932384}
{\sc {Evans}, R., {Lindner}, B., and {Shi}, Y.}
\newblock {Generating sudoku puzzles and its applications in teaching
  mathematics.}
\newblock {\em {Int. J. Math. Educ. Sci. Technol.} 42}, 5 (2011), 697--704.

\bibitem{zbMATH06322940}
{\sc {Fontana}, R.}
\newblock {Random Latin squares and Sudoku designs generation.}
\newblock {\em {Electron. J. Stat.} 8}, 1 (2014), 883--893.

\bibitem{HJ}
{\sc Horn, R.~A., and Johnson, C.~R.}
\newblock {\em {Topics in Matrix Analysis}}.
\newblock Cambridge University Press, Cambridge, 1991.

\bibitem{zbMATH05994814}
{\sc {Iv\'anyi}, A., and {N\'emeth}, Z.}
\newblock {List coloring of Latin and Sudoku graphs.}
\newblock In {\em {8th joint conference on mathematics and computer science,
  MaCS 2010, Kom\'arno, Slovakia, July 14--17, 2010. Selected papers}}. Gy\H
  or: NOVADAT, 2011, pp.~23--34.

\bibitem{zbMATH06172103}
{\sc {Jones}, S.~K., {Perkins}, S., and {Roach}, P.~A.}
\newblock {The structure of reduced Sudoku grids and the Sudoku symmetry
  group.}
\newblock {\em {Int. J. Comb.} 2012\/} (2012), 6 pages.

\bibitem{zbMATH05823702}
{\sc {Kanaana}, I., and {Ravikumar}, B.}
\newblock {Row-filled completion problem for Sudoku.}
\newblock {\em {Util. Math.} 81\/} (2010), 65--84.

\bibitem{zbMATH06092025}
{\sc {Keedwell}, A.~D.}
\newblock {Confirmation of a conjecture concerning orthogonal Sudoku and
  bimagic squares.}
\newblock {\em {Bull. Inst. Comb. Appl.} 63\/} (2011), 39--47.

\bibitem{semester}
{\sc {Klotz}, W., and {Sander}, T.}
\newblock {Wie kommt Sudoku zu ganzzahligen Eigenwerten?}
\newblock {\em {Math. Semesterber.} 57}, 2 (2010), 169--183.

\bibitem{zbMATH06619799}
{\sc {Lorch}, J.}
\newblock {Constructing ordered orthogonal arrays via sudoku.}
\newblock {\em {J. Algebra Appl.} 15}, 8 (2016), 19 pages.

\bibitem{zbMATH06328244}
{\sc {McGuire}, G., {Tugemann}, B., and {Civario}, G.}
\newblock {There is no 16-clue sudoku: solving the sudoku minimum number of
  clues problem via hitting set enumeration.}
\newblock {\em {Exp. Math.} 23}, 2 (2014), 190--217.

\bibitem{takingsud}
{\sc {Rosenhouse}, J., and {Taalman}, L.}
\newblock {\em {Taking Sudoku seriously. The math behind the world's most
  popular pencil puzzle.}}
\newblock Oxford University Press, 2011.

\bibitem{sudejc}
{\sc {Sander}, T.}
\newblock {Sudoku graphs are integral.}
\newblock {\em {Electron. J. Comb.} 16}, 1 (2009), research paper N25, 7 pages.

\bibitem{zbMATH06288901}
{\sc {Smith}, J.~D.}
\newblock {Palindromic and s\={u}doku quasigroups.}
\newblock {\em {J. Comb. Math. Comb. Comput.} 88\/} (2014), 85--94.

\end{thebibliography}

\end{document}